\documentclass[11pt]{amsart} 
\usepackage{amssymb}
\usepackage{verbatim}
\usepackage{hyperref}
\usepackage{color}

\usepackage{tikz}

\begin{document}
\setcounter{tocdepth}{1}

\newtheorem{theorem}{Theorem}    
\newtheorem{proposition}[theorem]{Proposition}
\newtheorem{corollary}[theorem]{Corollary}
\newtheorem{lemma}[theorem]{Lemma}
\newtheorem{sublemma}[theorem]{Sublemma}
\newtheorem{conjecture}[theorem]{Conjecture}
\newtheorem{claim}[theorem]{Claim}
\newtheorem{fact}[theorem]{Fact}
\newtheorem{observation}[theorem]{Observation}

\newtheorem{definition}{Definition}
\newtheorem{notation}[definition]{Notation}
\newtheorem{remark}[definition]{Remark}
\newtheorem{question}[definition]{Question}
\newtheorem{questions}[definition]{Questions}
\newtheorem{hypothesis}[definition]{Hypothesis}

\newtheorem{example}[definition]{Example}
\newtheorem{problem}[definition]{Problem}
\newtheorem{exercise}[definition]{Exercise}

 \numberwithin{theorem}{section}
 \numberwithin{definition}{section}
 \numberwithin{equation}{section}

\def\repair{\medskip\hrule\hrule\medskip}

\def\bff{\mathbf f}
\def\bE{\mathbf E}
\def\bF{\mathbf F}
\def\bK{\mathbf K}
\def\bP{\mathbf P}
\def\bx{\mathbf x}
\def\bi{\mathbf i}
\def\bk{\mathbf k}
\def\bt{\mathbf t}
\def\bc{\mathbf c}
\def\ba{\mathbf a}
\def\bw{\mathbf w}
\def\bh{\mathbf h}
\def\bn{\mathbf n}
\def\bg{\mathbf g}
\def\bc{\mathbf c}
\def\bs{\mathbf s}
\def\bp{\mathbf p}
\def\by{\mathbf y}
\def\bv{\mathbf v}
\def\be{\mathbf e}
\def\bu{\mathbf u}
\def\bm{\mathbf m}
\def\bxi{{\mathbf \xi}}
\def\bR{\mathbf R}
\def\by{\mathbf y}
\def\bz{\mathbf z}
\def\bfb{\mathbf b}
\def\bPhi{{\mathbf\Phi}}

\newcommand{\norm}[1]{ \|  #1 \|}

\def\scriptl{{\mathcal L}}
\def\scriptc{{\mathcal C}}
\def\scriptd{{\mathcal D}}
\def\scrapd{{\mathcal D}}
\def\scripts{{\mathcal S}}
\def\scriptq{{\mathcal Q}}
\def\scriptt{{\mathcal T}}
\def\scriptf{{\mathcal F}}
\def\scriptm{{\mathcal M}}
\def\calM{{\mathcal M}}
\def\scripti{{\mathcal I}}
\def\scriptr{{\mathcal R}}
\def\scriptb{{\mathcal B}}
\def\scripte{{\mathcal E}}
\def\scripta{{\mathcal A}}
\def\scriptn{{\mathcal N}}
\def\scriptv{{\mathcal V}}
\def\scriptz{{\mathcal Z}}
\def\scriptj{{\mathcal J}}
\def\scriptk{{\mathcal K}}
\def\scriptg{{\mathcal G}}
\def\scripth{{\mathcal H}}

\def\HLM{{\mathbb M}}
\def\mbbm{{\mathfrak M}}

\def\bk{\mathbf k}
\def\kernel{\operatorname{kernel}}
\def\dist{\operatorname{distance}\,}
\def\eps{\varepsilon}

\def\reals{\mathbb R}
\def\naturals{\mathbb N}
\def\integers{\mathbb Z}
\def\rationals{\mathbb Q}
\def\one{\mathbf 1}
\def\complex{{\mathbb C}\/}

\def\lt{{L^2}}

\def\three{\mathbf 3}
\def\four{\mathbf 4}

\def\bart{\bar t}
\def\barz{\bar z}
\def\barx{{\bar x}}
\def\bary{\bar y}
\def\barz{{\bar z}}
\def\bars{\bar s}
\def\barc{\bar c}
\def\baru{\bar u}
\def\barr{\bar r}

\def\distance{\operatorname{distance}}
\def\md{{\mathcal D}}

\def\lsharp{\Lambda^\sharp}
\def\lnatural{\Lambda^\natural}
\def\sS{{\mathcal S}}
\def\barsS{\overline{\mathbb S}}

\title {A class of singular bilinear maximal functions}

\author{Michael Christ}
\author{Zirui Zhou}

\address{
        Michael Christ\\
        Department of Mathematics\\
        University of California \\
        Berkeley, CA 94720-3840, USA}
\email{mchrist@berkeley.edu}

\address{
        Zirui Zhou\\
        Department of Mathematics\\
        University of California \\
        Berkeley, CA 94720-3840, USA}
\email{zirui\_zhou@berkeley.edu}

\begin{abstract}
Lebesgue space bounds $L^{p_1}(\reals^1) \times L^{p_2}(\reals^1) \to L^q(\reals^1)$
are established for certain maximal bilinear operators.
The proof combines a trilinear smoothing inequality with Calder\'on-Zygmund theory.
\end{abstract}

\thanks{Research of the first author was supported by NSF grant
DMS-1901413}

\date{March 30, 2022}

 \maketitle

 \section{Results}

 Let $I_0\subset\reals$ be an open interval. Let $\gamma:I_0\to\reals^2$ be real analytic
 with $\frac{d\gamma}{dt}$ vanishing nowhere.
 Write $\gamma(t) = (\gamma_1(t),\gamma_2(t))$.
 Let $\eta\in C^\infty_0(I_0)$  be a nonnegative, smooth, compactly supported auxiliary function.
 Consider bilinear operators of the form
 \begin{equation} 
 B_r(f_1,f_2)(x) = \int_{\reals^1} \prod_{j=1}^2 f_j(x+r\gamma_j(t))\,\eta(t)\,dt
 \end{equation}
 for $r\in(0,\infty)$.
Each $B_r$ maps a pair of functions of a single real variable to a function of one real variable.
$B_r(f_1,f_2)$ is the restriction to the diagonal in $\reals^1\times\reals^1$
 of an integral whose natural domain of definition is the full product space.
These forms are initially defined for continuous functions $f_j:\reals^1\to\complex$,
but their domains naturally contain appropriate Lebesgue classes.

Define associated maximal functions by
 \begin{align}
	 M^{\text{full}}(f_1,f_2)(x)
	 &= \sup_{r\in(0,\infty)} |B_r(f_1,f_2)(x)|
	 \\
	 \calM(f_1,f_2)(x)
	 &= \sup_{r\in 2^\integers} |B_r(f_1,f_2)(x)|
\end{align}
Both $M^{\text{full}}(f_1,f_2)(x)$
and $\calM(f_1,f_2)(x)$ are functions of $x\in\reals^1$.

A natural example is $\gamma(t) = (\cos(t),\sin(t))$, which leads to the maximal function
\begin{equation}
	\sup_{r\in 2^{\integers}} \big| \int_{S^1} f_1(x+ r y_1)\, f_2(x+ r y_2)\,d\sigma(y) \big|
\end{equation}
with $\sigma$ denoting arc length measure on the unit circle $S^1\subset\reals^2$.

Our main theorem requires three hypotheses.
Denote by $\gamma'_j$ the derivative of $\gamma_j$.
Define
\begin{equation} J(t) = \gamma'_1(t)-\gamma_2'(t). \end{equation}

\noindent {\bf Hypothesis 1.}\ 
$\gamma'_1,\gamma'_2$ are linearly independent over $\reals$. 
\newline {\bf Hypothesis 2.}\ 
There do not exist $a,b_1,b_2\in\complex$ 
with $a\ne 0$ and at least one $b_j\ne 0$ such that $\sum_{j=1}^2 b_je^{a\gamma_j(t)}$
is constant in $I_0$.
\newline {\bf Hypothesis 3.}\ 
\begin{equation} \label{hypothesis1}
|J(t)| + |J'(t)| \ne 0 \ \  \forall\,t\in I_0.
\end{equation}

The first hypothesis is equivalent to the assumption that the range of $\gamma$ is not contained
in any affine subspace of $\reals^2$. A consequence is that
all of the three functions $\gamma_1,\gamma_2,\gamma_1-\gamma_2$ are nonconstant.


To any pair of exponents $p_1,p_2\in[1,\infty]$ is associated the exponent
$q= q(p_1,p_2)\in[\tfrac12,\infty]$ defined by
\begin{equation} q^{-1} = p_1^{-1} + p_2^{-1}.\end{equation}

Boundedness of $M^{\text{full}}$ when $p_1,p_2>2$  is virtually immediate.
Indeed, denoting the one-dimensional Hardy-Littlewood maximal function by $\HLM$
and setting $\HLM_\tau f = (\HLM(|f|^\tau))^{1/\tau}$,
there is a uniform pointwise upper bound
$M^{\text{full}}(f_1,f_2)\le C_\tau \HLM_2(f_1)\cdot \HLM_\tau(f_2)$ for every $\tau >2$;
see Lemma~\ref{lemma:cauchy-schwarz}. Thus it is the extension to smaller exponents,
and in particular, to $q\le 1$, that is in question for the three maximal
operators $\calM,\,M^{\text{full}}_0,\,M^{\text{full}}$. 


Our main result concerns $\calM$. 

\begin{theorem}\label{thm:lacunary}
Let $I_0\subset\reals$ be a nonempty open interval.
Let $\varphi:I_0\to\reals^2$ be a real analytic mapping that satisfies the three hypotheses.
Let $\eta:I_0\to[0,\infty)$ be infinitely differentiable and have compact support in $I_0$.
For any $p_1,p_2>1$ there exists $C<\infty$ such that
\begin{equation}
\norm{\calM(f_1,f_2)}_{L^{q}(\reals^1)}
\le C \norm{f_1}_{L^{p_1}(\reals^1)}
\norm{f_2}_{L^{p_2}(\reals^1)}
\end{equation}
where $q = q(p_1,p_2)$.
\end{theorem}

The range of exponents $q$ for which the conclusion holds for some $(p_1,p_2)$ 
extends below $q=1$, indeed, to all $q>\tfrac12$.

Lacey \cite{lacey} has shown that if 
$\gamma_j(t) = c_j t$ for distinct constants $c_j$
then $\calM$ satisfies the indicated inequalities in the range $q(p_1,p_2)>\tfrac23$.
The nonlinearizable situation seems to require different techniques.

A well known result for maximal linear operators 
in the same spirit as Theorem~\ref{thm:lacunary} states that if $k<d$, if
$\gamma:\reals^{k}\to\reals^d$ is real analytic in a neighborhood of a compact set $K$, 
and if the range of $\gamma$
is not contained in any affine subspace of $\reals^d$,
then the maximal function
$\sup_{r\in 2^{\integers}} \int_K |f(x+r\gamma(t))|\,\eta(t)\,dt$
is bounded on $L^p(\reals^d)$. 
See \cite{stein+wainger_BAMS} for an introduction to this circle of ideas.
A prototypical example is
\[ \scriptm^0 f(x) = \sup_{r\in 2^\integers} \big| \int_{S^{d-1}} f(x+ry)\,d\sigma(y)\big|\]
where $x\in\reals^d$, $f:\reals^d\to\complex$, and $\sigma$ is 
surface measure on the unit sphere $S^{d-1}\subset\reals^d$.
A key fact is that
if $\eta\in C^\infty$ is supported in $K$, then the Fourier transform of the measure 
defined by $d\mu(t) = \eta(\gamma(t))\,dt$ 
satisfies $\widehat{\mu(\xi)} = O(|\xi|^{-\delta})$ as $|\xi|\to\infty$, for some $\delta>0$. 
A central element of our analysis, Theorem~\ref{thm:trilinearsmoothing}, states that 
under natural hypotheses, a bilinear analogue of this Fourier transform decay property holds.

We do not know whether $\calM$ map $L^1\times L^p$ for $p>1$, or even  to weak $L^1$ for $p=1$. 
Even for maximal linear operators such as $\scriptm^0$,
it remains an open question whether weak type $(1,1)$ inequalities hold. Certain partial results
are known \cite{christHardy}, \cite{STW}.

For those exponents $(p_1,p_2)\in(1,\infty)^2$ satisfying $q>1$, that is, $p_1^{-1} + p_2^{-1}<1$,
Theorem~\ref{thm:lacunary} is an easy of consequence of linear one-dimensional Calder\'on-Zygmund theory.
It suffices to treat the case $p_2=\infty$ --- for then interchanging the indices $j=1,2$ and interpolating
gives the general case --- and this case follows from Lemma~\ref{lemma:logloss}, below.

\begin{corollary} \label{cor:bicircular}
For $r\in\reals$ define
\[B_r(f_1,f_2)(x) = \int_0^{2\pi} |f_1(x+r\cos(\theta))\cdot f_2(x+r\sin(\theta))|\,d\theta.\]
Let $q = q(p_1,p_2)\in(0,\infty)$ satisfy $q^{-1} = p_1^{-1} + p_2^{-1}$ for $(p_1,p_2)\in[1,\infty]^2$.
For each $(p_1,p_2)\in(1,\infty]^2$ there exists $C<\infty$ such that
\begin{equation}
\norm{\sup_{r\in 2^{\integers}} |B_r(f_1,f_2)|}_{L^q(\reals^1)} \le C \prod_{j=1}^2 \norm{f_j}_{L^{p_j}(\reals^1)}.
\end{equation}
\end{corollary}

The quantity $J(\theta)$ defined above is equal to $\cos(\theta)+\sin(\theta)$
in this special case. It vanishes at $\theta = \tfrac{3\pi}4$ and at $-\tfrac\pi4$.
These points play distinguished roles in our analysis.

There have been several works concerning maximal operators associated to bilinear forms
\begin{equation} \label{2d-1} 
B_r(f_1,f_2)(x) = \int_{S^{2d-1}} |f_1(x+ry_1)\,f_2(x+ry_2)|\,d\sigma(y_1,y_2)
\end{equation}
with $(y_1,y_2)\in S^{2d-1}\subset \reals^d\times\reals^d$
and with $\sigma$ denoting surface measure on $S^{2d-1}$
for $d\ge 1$, with suprema taken over $r\in(0,\infty)$ and/or $r\in 2^\integers$.
After partial results by Barrionuevo, Grafakos, He, Honz\'{\i}k,
and Oliveira \cite{barrionuevo+etal} 
and by Heo, Hong, and Yang \cite{heo+hong+yang},
Jeong and Lee \cite{jeong+lee} characterized the full range of exponents for which
these maximal operators are bounded from $L^p(\reals^d)\times L^p(\reals^d)$ to $L^q(\reals^d)$
for each $d\ge 2$.
The more singular case of dimension $d=1$ was left open.
One starting point \cite{jeong+lee} is, for each $y_1$ in the unit ball,
to majorize the integral with respect to $y_2$ in the integral representation \eqref{2d-1}
by the $d$--dimensional spherical maximal function of $f_2$, evaluated at $x\in \reals^d$,
times $|f_1(x+ry_1)|$, multiplied by an appropriate scalar depending on $y_1$.
When $d-1=1$, however, this reduction breaks down.
Our analysis of $\calM$ for $d=1$ is rather different 
and relies on recent progress concerning multilinear smoothing inequalities.

One natural generalization to multilinear operators of higher degree is as follows.
Let $d\ge 1$.
Let $\sigma_d$ denote surface measure on the unit sphere $S^{d-1}\subset\reals^d$.
Let $e_j \in\reals^d$ be the $j$-th coordinate vector.
Define 
\[ \scriptm_d(f_1,\dots,f_d)(x) = \sup_{r\in 2^{\integers}} 
\int_{S^{d-1}} \prod_{j=1}^d |f_j(x+ r e_j\cdot y)|\,d\sigma(y)\]
for $x\in\reals^1$ and $f_j\in C^0(\reals^1)$.
We plan to analyze $\scriptm_d$ for $d\ge 3$ in a sequel.

The corresponding result for $M^{\text{full}}$ is as follows.
We also consider a variant $M^{\text{full}}_0$.
Let $\scripts\subset S^1$ be a closed subset that
does not contain any of the four points $(\pm1,0)$ and $(0,\pm 1)$.
Let $\sigma$ denote arc length measure on $S^1$.
Define
\[ M^{\text{full}}_0(f_1,f_2)(x)
= \sup_{r>0} \int_\scripts |f_1(x+r\cos(y))\,f_2(x+r\sin(y))|\,d\sigma(y).\]

\begin{theorem}\label{thm:full}
Let $\Omega_0$ be the closed convex hull of $\{(1, 0), (0, 0), (0, 1)\}$,
minus $\{(1,0),(0,1)\}$.
Let $\Omega$ be the interior of the convex hull of
$\{(0, 0), (0, {1\over 2}), ({1\over 2}, 0), ({1\over 2}, {1\over 2}) \}$,
minus $\{(\tfrac12,\tfrac12)\}$.

\noindent
(i) For any $(p_1^{-1},p_2^{-1})\in \Omega_0$, 
$M^{\text{full}}_0$ maps $L^{p_1}\times L^{p_2}$ boundedly to $L^q(p_1,p_2)$.

\noindent
(ii) For any $(p_1^{-1},p_2^{-1})\in \Omega$, 
$M^{\text{full}}$ maps $L^{p_1}\times L^{p_2}$ boundedly to $L^q(p_1,p_2)$.
\end{theorem}

Part (ii) of Theorem~\ref{thm:full} has been obtained independently by
Dosidis and Ramos \cite{dosidis_ramos}.
Those authors have also treated the generalizations
$\sup_{r\in(0,\infty)} \int_{S^{d-1}} \prod_{j=1}^d |f_j(x+re_j\cdot y)|\,d\sigma(y)$
of $M^{\text{full}}$ for $d\ge 3$.

\section{Single scale inequalities}

\subsection{Lebesgue norm bounds for the bilinear operator $\scriptb$}

Define
\begin{equation}
\scriptb(f_1,f_2)(x) = \int \prod_{j=1}^2 (f_j\circ\varphi_j)(t)\,\eta(t)\,dt.
\end{equation}
Here we analyze this basic bilinear operator, whose definition involves no supremum.

\begin{lemma} \label{lemma:mostbasic}
Suppose that $d\gamma/dt$ vanishes nowhere on $I_0$ and that
$\gamma$ satisfies hypothesis \eqref{hypothesis1}. Then
$\scriptb$ maps $L^{p_1}\times L^{p_2}$ boundedly to $L^{q(p_1,p_2)}$ for all $p_1,p_2\in(1,\infty]$.

If $J(t)$ vanishes nowhere on $I_0$ then
$\scriptb$ maps $L^1\times L^1$ boundedly to $L^1\cap L^{1/2}$.
\end{lemma}

Kenig and Stein \cite{kenig+stein} have observed that the second conclusion holds
when $\varphi_j$ are independent linear mappings.  
When $J$ vanishes nowhere, their analysis applies with no significant changes.
In this nonvanishing case, other results follow by interpolating this
bound for $L^1\times L^1$ with
trivial bounds $L^{p_1}\times L^{p_2}\to L^{q(p_1,p_2)}$
with $(p_1,p_2)$ equal to each of $(1,\infty)$, $(\infty,1)$, and $(\infty,\infty)$.

\begin{proof}
First consider the case in which $J$ vanishes nowhere. 
Assuming without loss of generality that $f_j\ge 0$, 
\[\norm{\scriptb(f_1,f_2)}_{L^1} 
= \int_{\reals} \int_{\reals} \prod_{j=1}^2 f_j(x+\gamma_j(t))\,\eta(t)\,dt\,dx.\]
The Jacobian determinant of the mapping $(x,t)\mapsto (x+\gamma_1(t),\,x+\gamma_2(t))$
	is $|\dot\gamma_1(t)-\dot\gamma_2(t)| = |J(t)|$.
From this, the $L^1\times L^1\to L^1$ bound follows. 

For the $L^{1/2}$ bound in the case in which $J$ does not vanish,
we localize and invoke the $L^1$ bound.
Partition $\reals^1$ into intervals $I_n$ of lengths equal to $1$.
Write $\sum_{n,n_1,n_2}^*$ to indicate a sum over all triples $(n,n_1,n_2)$
satisfying $|n_1-n|+|n_2-n|\le C$ for a sufficiently large constant $C$.
Then
\begin{align*}
\int_{\reals} \scriptb(f_1,f_2)(x)^{1/2}\,dx
&	= \sum_n \int_{I_n} \scriptb(f_1,f_2)(x)^{1/2}\,dx
\\&
= \sum_{n,n_1,n_2}^* 
	\int_{I_n} \scriptb(f_1\one_{I_{n_1}},f_2\one_{I_{n_2}})^{1/2}\,dx
\\&
\le C \sum_{n,n_1,n_2}^* 
	\left(\int_{I_n} \scriptb(f_1\one_{I_{n_1}},f_2\one_{I_{n_2}})\,dx \right)^{1/2}
\\&
\le C \sum_{n,n_1,n_2}^* 
\prod_{j=1}^2 \norm{f_j \one_{I_{n_j}}}_1^{1/2}
\\&
	\le C \sum_{n} 
	\prod_{j=1}^2 \norm{f_j \one_{I_{n}^*}}_1^{1/2}
\end{align*}
where $I_{n}^*$ is the interval of length $C$ with the same center as $I_n$,
for a sufficiently large constant $C$.
By the Cauchy-Schwarz inequality, the last line is $O(\norm{f_1}_1^{1/2}\norm{f_2}_1^{1/2})$.

Next consider the case in which $J$ may vanish, but only to first order.
By introducing a partition of unity and applying the result for nonvanishing $J$
proved above, and by making a change of variables with respect to $t$,
we may assume that $J$ vanishes only at $t=0$,
and thus that $|J(t)|$ is comparable to $t$.
We may also assume that $\gamma(0)=0$.

Let $r\in(0,1]$ be arbitrary
and consider $B_r(f_1,f_2)(x) = \int_{r\le |t| \le 2r}
\prod_{j=1}^2 (f_j\circ\varphi_j)(x,t)\,dt$.
Partition $\reals^1$ into intervals $I_n$ of length $r$.
Proceeding as above, obtain
\begin{align*}
\int_{\reals} B_r(f_1,f_2)(x)^{1/2}\,dx
&\le C' \sum_{n,n_1,n_2}^* r^{1/2}  
\Big(
\int_{I_n} \scriptb(f_1\one_{I_{n_1}},f_2\one_{I_{n_2}})\,dx
\Big)^{1/2} 
\\&
\le C'' r^{1/2} \sum_{n,n_1,n_2}^* 
\big( r^{-1} \prod_{j=1}^2 \norm{f_j \one_{I_{n_j}}}_1\big)^{1/2}
\end{align*}
for certain constants $C',C''<\infty$.
This again is $O(\norm{f_1}_1^{1/2}\norm{f_2}_1^{1/2})$, uniformly in $r$.
On the other hand, $B_r$ is $O(1)$ from $L^1\times L^\infty$ to $L^1$
and likewise from $L^\infty\times L^1$ to $L^1$,
and is $O(r)$ from $L^\infty\times L^\infty$ to $L^\infty$.
By interpolating these bounds, taking $r = 2^{-k}$ for $k\in\naturals$,
and summing over $k$, we obtain the desired conclusion
for $L^{p_1}\times L^{p_2}$ whenever $p_1,p_2>1$.
\end{proof}

\noindent{\bf Remark.}\ 
The hypothesis that $J,J'$ do not vanish simultaneously is necessary 
for $\calM$ to be bounded for the full range of exponents
indicated in Theorem~\ref{thm:lacunary}. Indeed, it is necessary even for
$\scriptb$ to satisfy those inequalities.
That is, if there exists $\bart$ satisfying $\eta(\bart)\ne 0$,
$J(\bart)= 0$, and $J'(\bart)= 0$,
then whenever $p_1^{-1}+p_2^{-1} >\tfrac32$,
$\scriptb$ fails to map $L^{p_1}\times L^{p_2}$ boundedly to $L^{q(p_1,p_2)}$.

This can be seen by choosing $f_j$ to be the indicator
function of an interval of length $\delta$ centered at $\gamma_j(\bart)$ for $j=1,2$.
Then for all sufficiently small $\delta$, $\scriptb(f_1,f_2)(x)\gtrsim \delta$ 
on an interval of length $\gtrsim\delta^{1/3}$ centered at $0$.
Then $\norm{\scriptb(f_1,f_2)}_q\gtrsim \delta\cdot\delta^{3/q}$,
while $\norm{f_j}_{p_j}\asymp \delta^{1/p_j}$.
When $p_1^{-1}+p_2^{-1} >\tfrac32$, 
the ratio $\norm{\scriptb(f_1,f_2)}_q/\norm{f_1}_{p_1}\norm{f_2}_{p_2}$
tends to infinity as $\delta$ tends to zero. \qed

\noindent {\bf Remark.}\ 
Consider forms
\[\scriptb(f_1,f_2,\dots,f_n) = \int_\reals \prod_{j=1}^n f_j(x+\gamma_j(t))\,\eta(t)\,dt\]
with $n\ge 3$. The above analysis does not apply for $\bff = (f_1,\dots,f_n) \in L^1\times \cdots \times L^1$.
Under mild hypotheses, results for $L^1\times L^1\times L^\infty \cdots\times L^\infty$,
with exactly two factors in $L^1$ and $n-2$ factors in $L^\infty$,
can be obtained by the above reasoning. The indices $j$ can be permuted, and the resulting
bounds interpolated.
Can one go beyond results obtained in this way?

The case $n=3$ with each $\gamma_j(t) = a_j t$ linear was investigated in \cite{elementary!},
where it was shown that for $(a_1,a_2,a_3) = (1,-1,\alpha)$,
certain nontrivial bounds hold for rational $\alpha$, but depend on 
the Diophantine character of $\alpha$, and break down for irrational $\alpha$. 
The authors are not aware of work for nonlinear $\varphi_j$ for $n = 3$
that goes beyond results obtainable from $L^1\times L^1\times L^\infty$ bounds.  \qed

\subsection{Trilinear smoothing inequality}\label{subsection:Qtsi}

Let $\varphi_j:\reals^2\to\reals^1$ be real analytic mappings.
Let $\eta\in C^\infty_0(\reals^2)$ be infinitely differentiable and have compact support.
Consider a trilinear form
\begin{equation} \label{trilinearform}
\scriptt(\bff) = \int_{\reals^2} \prod_{j=0}^2 (f_j\circ\varphi_j)(x)\,\eta(x)\,dx
\end{equation}
acting on ordered triples $\bff = (f_0,f_1,f_2)$ of functions $f_j:\reals^1\to\complex$.

\begin{theorem} \label{thm:trilinearsmoothing}
Let $U$ be a connected neighborhood of the support of $\eta$.
Let $\varphi_j:U\to\reals^1$ be real analytic.
Assume that for any $i\ne j\in\{0,1,2\}$,
$\det(\nabla\varphi_i,\nabla\varphi_j)$ does not vanish identically 
in any nonempty open set.
Assume that for any nonempty connected open subset $U'\subset U$,
for any $\bg\in C^\omega(\Phi(U'))$ that satisfies
$\sum_{j=0}^2 (g_j\circ\varphi_j) \equiv 0$
in $U'$, each $g_j$ is constant in $\varphi_j(U')$.

Then there exist $p<\infty$, $\sigma<0$, and $C<\infty$
such that for all Lebesgue measurable functions 
$\bff  = (f_0,f_1,f_2)\in (L^p\times L^p\times L^p)(\reals^1)$,
the integral defining $\scriptt(\bff)$ converges absolutely and 
\begin{equation}
|\scriptt(\bff)| \le C \prod_{j=0}^2 \norm{f_j}_{W^{p,\sigma}}.
\end{equation}
\end{theorem}

For the nondegenerate case in which $\nabla\varphi_i,\nabla\varphi_j$
are everywhere linearly independent for all $i\ne j$,
this is proved in reference \cite{trilinear}, by a roundabout argument based on a somewhat
more complicated theorem.  
This linear independence hypothesis is not satisfied in the application to Corollary~\ref{cor:bicircular}. 
We will sketch a proof for the general case that is
implicit in a combination of the works \cite{trilinear}, \cite{CDR}, and \cite{quadrilinear}.

The next lemma connects the main hypothesis of Theorem~\ref{thm:trilinearsmoothing}
to the hypotheses of Theorem~\ref{thm:lacunary}.

\begin{lemma} \label{lemma:curvaturehypothesis}
Let $d\ge 2$.
Let $U\subset\reals^{d-1}$  be a connected nonempty open set.
Let $\gamma = (\gamma_1,\gamma_2,\dots,\gamma_{d}):U\to\reals^{d}$ be real analytic.
Let $\varphi_j(x,t) = x+\gamma_j(t)$ for $j\in\{1,2,\dots,d\}$ and $\varphi_0(x,t)=x$.
Assume that the Jacobian determinant of the mapping $(x,t_1,\dots,t_{d-1})
\mapsto \Phi(x,t_1,\dots,t_{d-1}) = (\varphi_i(x,t): 1\le i\le d)\in\reals^{d}$ does not vanish identically
on $\reals\times U$.
Assume also  that $\complex\cup\{\gamma_j: j\in\{1,2,\dots,d\}\}$ is linearly independent over $\complex$.

The following two conditions are equivalent. 
\newline (a)
There exist a nonempty connected open set $W\subset \reals\times U$
and nonconstant real analytic functions $f_j:\varphi_j(W)\to\complex$
satisfying $\sum_{j=0}^d (f_j\circ\varphi_j)(x,t)\equiv 0$ in $W$.
\newline
(b) Either 
\newline \qquad \phantom{a}\qquad (i) $\gamma_i-\gamma_j$ is constant for some pair of indices $i\ne j\in\{1,2,\dots,d\}$,
\newline
or 
\newline \phantom{a}\qquad (ii) There exists $(a,b_1,b_2,\dots,b_d)\in\complex^{d+1}$
with $a\ne 0$ and at least one $b_j$ nonzero, such that
$\sum_{j=1}^d b_j e^{a\gamma_j(t)}$ is constant on $U$.
\end{lemma}

The linear independence hypothesis says that if $b_j\in\complex$ and $\sum_{j=1}^d b_j\gamma_j(t)$
is constant in some nonempty open set, then each $b_j=0$.

\begin{proof}
If there exist such $a,b_j$,
define $b_0 = -\sum_{j=1}^d b_j e^{a\gamma_j(t)}$, which is a constant by hypothesis.
Define $f_j(y) = b_j e^{ay}$ for each index $j\in\{0,1,2,\dots,d\}$.
Then 
\[ \sum_{j=0}^d (f_j\circ\varphi_j)(x,t)
=e^{ax} \big(b_0 +  \sum_{j=1}^d b_je^{a\gamma_j(t)}\big)
\equiv 0 \ \text{ on $\reals\times U$,} \]
and at least one function $f_j$ is nonconstant.

If there exist $i\ne k\in\{1,2,\dots,d\}$ for which
$\gamma_k-\gamma_i$ is identically equal to a constant $c\in\reals$, 
then define $f_j\equiv 0$ for every $j\notin\{i,k\}$, 
define $f_i$ to be an arbitrary nonconstant function,
and define $f_k(y)= -f_i(y-c)$. 
Then $\sum_{j=0}^d f_j\circ\varphi_j$ vanishes identically.

To prove the converse,
assume that the Jacobian determinant of $\Phi$ does not vanish identically,
and moreover that conclusion (i) does not hold; thus $\gamma_i-\gamma_j$ is nonconstant whenever 
$i\ne j\in\{1,\dots,d\}$.
Suppose that $f_j$ are real analytic,
that $\sum_{j=0}^d (f_j\circ\varphi_j)\equiv 0$ in a connected open subset $W\subset \reals \times U$,
and that (after permuting the indices if necessary) $f_d$ is nonconstant.
We must show that each $f_j$ is constant in $\varphi_j(W)$. 

There exists $\barz = (\barx,\bart)\in W$
such that $f'_d(\varphi_d(\barz))\ne 0$, and such that the Jacobian determinant $\det(D\Phi)$ does not vanish at $\barz$.
By hypothesis, $\nabla\gamma_d$ does not vanish identically.
Therefore by making an arbitrarily small perturbation of $\bart$ and a real analytic change of variables, 
we may assume that $\gamma_j(t)\equiv t_j$ for all $j\in\{1,2,\dots,d-1\}$,
$\nabla\gamma_d(\bart)\ne 0$, and $f'_d(\varphi_d(\barz))\ne 0$.


Differentiate with respect to $t$ to obtain
\begin{equation}\label{diffonce}
f'_j(x+\gamma_j(t)) + \frac{\partial\gamma_d}{\partial t_j}(t)\,f'_d(x+\gamma_d(t))\equiv 0
\end{equation}
in a neighborhood of $\barz$
for every $j\in\{1,2,\dots,d-1\}$.
	Therefore the ratio \[ \frac{f'_j(x)}{f'_d\big(x+[\gamma_d(t)-\gamma_j(t)]\big)} \] is locally independent of $x$.
Since $\gamma_d-\gamma_j$ is nonconstant by assumption, 
$(x,u)\mapsto \frac{f'_j(x)}{f'_d(x+u)}$ is locally a function of $u$ alone.
This implies that there exist $b_j,a\in\complex$ such that $f'_j(y)\equiv b_je^{ay}$,

for all $y$ in a neighborhood of $\varphi_j(\barz)$, for each $j\in\{1,2,\dots,d-1\}$.
Moreover, $b_d\ne 0$, since $f_d$ is assumed to be nonconstant.

By differentiating the relation 	
$\sum_{j=0}^d (f_j\circ\varphi_j)(x,t)\equiv 0$ with respect to $x$,
we find that
$\sum_{j=0}^d (f'_j\circ\varphi_j)(x,t)\equiv 0$.
Therefore
\begin{equation} \label{arelation} \sum_{j=1}^d b_j e^{a\gamma_j(t)} \equiv e^{-ax}f'_0(x)\end{equation}
in a nonempty open set. If $a\ne 0$, specializing to any particular value of $x$ reveals that the left-hand
side is contant, which is conclusion (ii).

If $a=0$ then $f'_j$ is constant for each $j\ge 1$, and by \eqref{arelation}, $f'_0$ is likewise constant.
Therefore $f_j(y) = b_j y+c_j$ for some constants $b_j,c_j\in\complex$, for every $j\in\{0,1,\dots,d\}$.
Thus
\[ b_0 x  + \sum_{j=1}^d b_j\cdot (x+\gamma_j(t)) \equiv -c\]
in some nonempty connected open set, for some $c\in\complex$.
Therefore $\sum_{j=1}^d b_j\gamma_j(t)$ is constant in a nonempty connected open set,
contradicting the hypothesis that $\complex\cup \{\gamma_j\}$ is linearly independent over $\complex$.
\end{proof}

\section{Operator decomposition and some reductions}

Choose an auxiliary function $\phi\in\scripts(\reals^1)$ 
such that $\widehat{\phi}$ has compact support
and satisfies $\widehat{\phi}\equiv 1$ in a neighborhood of $0$.
Define operators $P_k,Q_k$ by
\begin{align*}
	\widehat{P_k f}(\xi) = \widehat{f}(\xi) (2^{-k}\xi)
\end{align*}
and $Q_k = P_{k+1}-P_k$.
We will exploit the identity $I = P_k + \sum_{n=1}^\infty Q_{k+n}$, where $I$ is the identity
operator; this identity holds on $L^p(\reals)$ for each $k\in\integers$ and $p\in(1,\infty)$,
with convergence in the strong operator topology.
Since $\widehat{\psi}$ has compact support and vanishes identically in a neighborhood of $0$,
there exists $\tilde\psi\in\scripts$ 
and so that $\widehat{\tilde\psi}\equiv 1$ on the support of $\widehat{\psi}$
and $\widehat{\tilde\psi}(0)=0$.
Define $\tilde Q_k$ by $\widehat{\tilde Q_k f}(\xi) = \widehat{f}(\xi) \widehat{\tilde\psi}(2^{-k}\xi)$.
Then $\tilde Q_k\circ Q_k\equiv Q_k$.

For each $k\in\integers$ and $\bn = (n_1,n_2) \in \{0,1,2,\dots\}^2$ define
\begin{align}
	A_k(f_1,f_2)(x) & = \int \prod_{j=1}^2 f_j(x+2^{-k}\gamma(t))\,\eta(t)\,dt.
	\\
	T_{\bn,k}(f_1,f_2) &= A_k(Q_{k+n_1}f_1,Q_{k+n_2}f_2)
	\\
	M_{\bn}(f_1,f_2)(x) &= \sup_{k\in\integers} |A_k(Q_{k+n_1}f_1,Q_{k+n_2}f_2)(x)|
	\\
	S_{\bn}(f_1,f_2) &=  \sum_{k\in\integers} |T_{\bn,k}(f_1,f_2)|.
\end{align}
Then for any $f_1,f_2$,
\begin{align*}
	\calM(f_1,f_2)
	&\le 
	 \sup_{k\in\integers}
	|A_k(P_kf_1,P_k f_2)|
	+ \sup_{k\in\integers}
	|A_k(P_kf_1,(I-P_k) f_2)|
	\\ &\qquad\qquad
	+ \sup_{k\in\integers}
	|A_k((I-P_k)f_1,P_k f_2)|
	+ \sum_{\bn\in\naturals^2} M_{\bn}(f_1,f_2)
	\\
	&\le 
	 3\sup_{k\in\integers} |A_k(P_kf_1,P_k f_2)|
	+ \sup_{k\in\integers}
	|A_k(P_kf_1,f_2)|
	\\ &\qquad\qquad
	+ \sup_{k\in\integers}
	|A_k(f_1,P_k f_2)|
	+ \sum_{\bn\in\naturals^2} M_{\bn}(f_1,f_2).
\end{align*}

We will use the following well known variant of the Hardy-Littlewood maximal theorem.

\begin{lemma} \label{lemma:logloss}
Let $I\in\reals^1$ be a bounded interval with center $z\in\reals^1$.
Define 
\begin{equation}
\scriptn_I f(x) = \sup_{r\in 2^{\integers}} |I|^{-1} \int_I |f(x-ry)|\,dy.
\end{equation}
Then $\scriptn_I$ is of weak type $(1,1)$ with bound $O( \log(2+|z||I|^{-1}))$.
For each $p\in(1,\infty)$,
$\scriptn_I$ maps $L^p(\reals^1)$ to $L^p(\reals^1)$
	with operator norm $O\big((\log(2+|z||I|^{-1}))^{1/p}\big)$.
\end{lemma}

This is for instance a simple consequence of the Calder\'on-Zygmund decomposition.
\qed

\begin{lemma}
The maximal operator $N(f_1,f_2) = \sup_{k\in\integers} |A_k(P_kf_1,f_2)|$
maps $L^{p_1}\times L^{p_2}$ boundedly to $L^{q(p_1,p_2)}$
whenever $p_1,p_2\in(1,\infty]$.
\end{lemma}

\begin{proof}
Let $\HLM$ denote the Hardy-Littlewood maximal function, acting on functions with domain $\reals^1$.
It is elementary that $|\scriptb(P_0f_1,f_2)|\le C\HLM(f_1)\cdot \HLM(f_2)$
if $\frac{d\gamma_2}{dt}$ never vanishes. 
Therefore $\sup_k |A_k(P_kf_1,f_2)| \le C\HLM(f_1)\cdot\HLM(f_2)$ as well,
by scale invariance.

Since $\gamma\in C^\omega$,
since the support of $\eta$ is compact, and since the range of $\gamma$
is not contained in any line in $\reals^2$,
the general case reduces to the subcase in which there is a single 
$t$ at which $\frac{d\gamma_2}{dt}(t)$ vanishes.
Then
\[\scriptb(P_0f_1,f_2)(x) = \int K(x-y_1,x-y_2)\,f_1(y_1)\,f_2(y_2)\,d\by\]
where $K$ has compact support and 
\[ |K(\by)| \le C|y_2-u|^{-\tau}\]
for some $u\in\reals^1$ and $\tau <1$.

The conclusion follows from Lemma~\ref{lemma:logloss} via a simple decomposition.
\end{proof}

The same result applies to the maximal operators $\sup_k |A_k(f_1,P_kf_2)|$
and (more simply) $\sup_k |A_k(P_kf_1,P_kf_2)|$.
The same reasoning also shows that $\calM$
maps $L^{p_1}\times L^{p_2}$ to $L^{q(p_1,p_2)}$
whenever $p_1=\infty$ or $p_2=\infty$.
Therefore in order to prove our main result, it suffices to 
prove that
for each $(p_1,p_2)\in(1,\infty)^2$ 
there exists $\delta = \delta(p_1,p_2)>0$ such that
for every $\bn\in\naturals^2$,
$M_\bn$ maps $L^{p_1}\times L^{p_2}$ to
$L^{q(p_1,p_2)}$ with bound $O(2^{-\delta|\bn|})$.

This will be a consequence of two inequalities, which will be proved below:
\begin{lemma} \label{lemma:Ltwo}
There exist $\delta>0$ and $C<\infty$
such that for all $f_1,f_2\in C^0\cap L^2(\reals^1)$,
\begin{equation}
	\norm{S_\bn(f_1,f_2)}_1 \le C2^{-\delta|\bn|}\prod_{j=1}^2\norm{f_m}_2.
\end{equation}
\end{lemma}

\begin{lemma} \label{lemma:Lone}
There exists $C<\infty$ such that for every $\bn\in\naturals^2$
and $f_1,f_2\in L^1(\reals^1)$,
\begin{equation}
	\norm{M_\bn(f_1,f_2)}_{L^{1/2,\infty}} \le C(1+|\bn|^2)\prod_{j=1}^2\norm{f_j}_1.
\end{equation}
\end{lemma}

\section{Proof of Lemma~\ref{lemma:Ltwo}}

For any cutoff function $\tilde\eta\in C^\infty_0(\reals^1)$,
the trilinear form
\begin{equation}
\scriptt(\bff) = \int f_0(x) \,\scriptb(f_1,f_2)(x)\,\tilde\eta(x)\,dx,
\end{equation}
acting on $\bff = (f_0,f_1,f_2)$,
satisfies the hypotheses of Theorem~\ref{thm:trilinearsmoothing}.
Therefore (by a simple localization argument) there exists $s<0$ such that
\begin{equation}
\norm{\scriptb(f_1,f_2)}_1\le C \prod_{j=1}^2\norm{f_j}_{W^{2,s}}
\end{equation}
for all $f_j\in L^2(\reals^1)$. 
Therefore
\begin{equation}
	\norm{\scriptb(Q_{n_1}f_1,Q_{n_2}f_2)}_1
	\le C 2^{-|s|\max(n_1,n_2)} \prod_{j=1}^2\norm{f_j}_2
	\le C 2^{-\delta|\bn|} \prod_{j=1}^2\norm{f_j}_2
\end{equation}
uniformly in $\bn$, with $\delta = |s|/2>0$.
By scaling, 
\begin{equation}
\norm{A_k(Q_{k+n_1}f_1,Q_{k+n_2}f_2)}_1
\le C 2^{-\delta|\bn|} \prod_{j=1}^2\norm{f_j}_2
\end{equation}
uniformly for all $k,\bn$.

By writing
$Q_i = \tilde Q_i\circ Q_i$  we obtain
\begin{align*}
	\norm{S_\bn(f_1,f_2)}_1
	& =  \norm{\sum_k T_{\bn,k}(f_1,f_2)}_1
	\\& \le  \sum_k \norm{T_{\bn,k}(f_1,f_2)}_1
	\\& =  \sum_k \norm{T_{\bn,k}(\tilde Q_{k+n_1}f_1,\tilde Q_{k+n_2}f_2)}_1
	\\& \le C2^{-\delta|\bn|} \sum_k \prod_{j=1}^2\norm{\tilde Q_{k+n_j}f_j}_2
	\\& \le C2^{-\delta|\bn|} \prod_{j=1}^2 (\sum_k \norm{\tilde Q_{k+n_j}f_j}_2^2)^{1/2}
	\\& \le C2^{-\delta|\bn|} \prod_{j=1}^2  \norm{f_j}_2^2.
\end{align*}

\section{Proof of Lemma~\ref{lemma:Lone}}

Lemma~\ref{lemma:Lone} follows from a straightforward consequence of a combination 
of multilinear Calder\'on-Zygmund theory with Lemma~\ref{lemma:mostbasic}.
Assume without loss of generality that $\norm{f_j}_1=1$.
Let $\alpha>0$. The inequality to be proved states that
\begin{equation} \label{distribution} |\{x: M_\bn(f_1,f_2)(x)>\alpha\}| 
\le C|\bn|^2 \alpha^{-1/2}.\end{equation}

Apply the Calder\'on-Zygmund decomposition to each $f_j$ at height $c_0\alpha^{1/2}$
for a certain small constant $c_0$.
Thus $f_j = g_j+h_j$ where $\norm{g_j}_\infty \le c_0\alpha$,
while
$h_j = \sum_\beta h_{j,\beta}\one_{I_{j,\beta}}$,
each $I_{j,\beta}\subset\reals^1$ is a dyadic interval of positive, finite length $|I_{j,\beta}|$,
$\norm{h_{j,\beta}}_1\le C \alpha|I_{j,\beta}|$,
and $\int h_{j,\beta}=0$.

\subsection{Contribution of $(h_1,h_2)$} \label {subsection:h1h2}
Let $I_{j,\beta_j}^*$ be the interval concentric with $I_{j,\beta_j}$ satisfying
$|I_{j,\beta_j}^*| = 4|I_{j,\beta_j}|$.
Define $\scripte\subset\reals^1$ by
\[ \scripte = (\bigcup_{\beta_1} I_{1,\beta_1}^*)
\cup (\bigcup_{\beta_2} I_{2,\beta_2}^*),\]
which satisfies 
\begin{equation}
|\scripte| = O(\alpha^{-1/2}).
\end{equation}
We will show that
\begin{equation}
\int_{\reals^1\setminus\scripte} M_\bn(h_1,h_2)^{1/2} 
\le C|\bn|^2.
\end{equation}
From this and the bound $|\scripte| = O(\alpha^{-1/2})$ it follows, by Chebyshev's inequality, that
$|\{x: M_\bn(h_1,h_2)(x)>\alpha/4\}|$ is majorized by $C|\bn|^2 \alpha^{-1/2}$.

Define
\[ h_j^i = \sum_{|I_{\beta_j}| = 2^{-i}} h_{j,\beta_j}.\]

Majorize
\begin{align*}
	\sup_{k\in\integers} |A_k(Q_{k+n_1}h_1,\,Q_{k+n_2}h_2)|^{1/2}
	&\le
	\sup_k \big(\sum_{i_1\in\integers} \sum_{i_2\in\integers} 
	|A_k(Q_{k+n_1}h_1^{i_1},\,Q_{k+n_2}h_2^{i_2})|\big)^{1/2}
\\	&\le \sup_k
	\sum_{i_1,i_2} |A_k(Q_{k+n_1}h_1^{i_1},\,Q_{k+n_2}h_2^{i_2})|^{1/2}
	\\&
	\le \sum_{i_1,i_2} \sum_k |A_k(Q_{k+n_1}h_1^{i_1},\,Q_{k+n_2}h_2^{i_2})|^{1/2}
\end{align*}
so that 
\[ 
\int_{\reals^1\setminus\scripte} M_\bn(h_1,h_2)^{1/2} 
\le
\sum_{i_1,i_2\in\integers}\, 
\sum_{k\in\integers}\,
\int_{\reals\setminus\scripte}
|A_k(Q_{k+n_1}h_1^{i_1},\,Q_{k+n_2}h_2^{i_2})|^{1/2}.\]

\begin{lemma}
Uniformly for all $k,\bn,i_1,i_2$,
\begin{equation}
	\int_{\reals^1\setminus\scripte} |A_k(Q_{k+n_1} h_1^{i_1},\,Q_{k+n_2} h_2^{i_2})|^{1/2}
	\lesssim
	\min_{j=1,2} \min\big(2^{i_j-k},\, 2^{(k+n_j-i_j)/2},\, 1\big)\,
	\prod_{l=1}^2 \norm{h_l^{i_l}}_1^{1/2}.
\end{equation}
\end{lemma}

\begin{proof}[Outline of proof]
\begin{equation} \label{CZbound}
\int_{\reals^1\setminus\scripte} |A_k(Q_{k+n_1} h_1^{i_1},\,Q_{k+n_2} h_2^{i_2})|^{1/2}
	\le C\prod_{j=1}^2 \norm{Q_{k+n_j} h_j^{i_j}}_1^{1/2}
\end{equation}
by Lemma~\ref{lemma:mostbasic}. Since $Q_{l}$ is defined by convolution
with a function whose $L^1$ norm is finite and independent of $l$,
it follows that
\[\int_{\reals^1\setminus\scripte} |A_k(Q_{k+n_1} h_1^{i_1},\,Q_{k+n_2} h_2^{i_2})|^{1/2}
\le C \prod_{j=1}^2 \norm{h_j}_1.\]

The other two bounds are obtained by combining \eqref{CZbound} with alternative bounds for $\norm{Q_lh_j^{i}}_1$.
When $i_j> k+n_j$,
a factor $2^{(k+n_j-i_j)/2}$ arises from combining this inequality with the bound
$\norm{Q_{l}h_j^i}_1 \lesssim \min(1,2^{l-i})\norm{h_j^i}_1$. 
This follows from
\[ \norm{Q_{l}h_{j,\beta_j}}_1 \lesssim \min(1,2^{l-i})\norm{h_j^i}_1 
\text{ whenever $|\beta_j| = 2^{-i}$,}\]
which is a routine consequence of the moment condition $\int h_{j,\beta_j}=0$
and the fact that each $h_{j,\beta_j}$ is supported on an interval of length $2^{-i}$.

Next, consider
$A_k(Q_{k+n_1} h_1^{i_1},\,Q_{k+n_2} h_2^{i_2})(x)$
when $k>i_1$ and $x\notin\scripte$. This is a sum of terms
$A_k(Q_{k+n_1} h_{1,\beta_1},\,Q_{k+n_2} h_{2,\beta_2})(x)$
where $|I_{j,\beta_j}| = 2^{-i_j}$.
Recall that
$Q_l$ is defined by convolution with $\psi_l(y) = 2^l\psi(2^l y)$, for a certain Schwartz function $\psi$.

Since $x\notin\scripte$, $x$ lies at distance greater than $2\cdot 2^{-i_1}$ from  $I_{1,\beta_1}$. 
Therefore since $2^{-k}\le 2^{-i_1}$, the distance from $x+2^{-k}\cos(t)$ to $I_{1,\beta_1}$ is $\ge 2^{-i_1}$
for any $t\in[0,2\pi]$.
As a result,
\[ A_k(Q_{k+n_1} h_1^{i_1},\,Q_{k+n_2} h_2^{i_2})(x)
=  A_k(\tilde\psi*h_1^{i_1},\,Q_{k+n_2} h_2^{i_2})(x)\]
where $\tilde\psi(y) = \psi_{k+n_1}(y)\one_{|y|\ge 2^{-i_1}}$.
This convolution kernel $\tilde\psi$ satisfies
\[ \norm{\psi_{k+n_1}(y)\one_{|y|\ge 2^{-i_1}}}_1 = O(2^{-(k+n_1-i_1)N})\]
for every $N<\infty$ since $\psi(y) = O(1+|y|)^{-N}$.
Therefore
\begin{align*} \norm{\tilde\psi*h_1^{i_1}}_{L^1(\reals\setminus\scripte)}
&\le \sum_{|I_{1,\beta_1}| = 2^{-i_1}}
\norm{\tilde\psi*h_{1,\beta_1}}_{L^1(\reals\setminus\scripte)}
\\&
\le \sum_{|I_{1,\beta_1}| = 2^{-i_1}}
\norm{\tilde\psi}_1\norm{h_{1,\beta_1}}_1
\\&
\le O(2^{-(k+n_1-i_1)N}) \cdot \norm{h_1^{i_1}}.
\end{align*}
Interchanging the roles of the indices $j=1,2$
yields a bound
$O(2^{-(k+n_2-i_2)N}) \cdot \norm{h_2^{i_2}}$
when $k+n_2>i_2$.
\end{proof}

\begin{lemma}
	There exists $C<\infty$ such that for every $\bn\in \naturals^2$
	and every $h_l = \sum_{i_l\in\integers} h_l^{i_l}$ as above,
\begin{equation}
	\sum_{k\in\integers}
	\sum_{i_1\in\integers} \sum_{i_2\in\integers}
	\min_{l=1,2} \min\big(2^{i_l-k},\, 2^{(k+n_l-i_l)/2},\, 1\big)\,
	\prod_{j=1}^2 \norm{h_j^{i_j}}_1^{1/2}
	 \lesssim |\bn|^2  \prod_{j=1}^2 \norm{h_j}_1^{1/2}.
\end{equation}
\end{lemma}

\begin{proof} 
\[ \sum_{k\in\integers}
\min_{l=1,2} \min\big(2^{i_l-k},\, 2^{(k+n_l-i_l)/2},\, 1\big)
\le C|\bn| \min(1,2^{-(|i_1-i_2|-|\bn|)/2})
\]
uniformly for all $(i_1,i_2)$.
Therefore by the Cauchy-Schwarz inequality,
\begin{align*}
	\sum_{i_1,i_2} \sum_{k\in\integers}
	\min_{l=1,2} 
	&\min\big(2^{i_l-k},\, 2^{(k+n_l-i_l)/2},\, 1\big)\,
	\prod_{j=1}^2 \norm{h_j^{i_j}}_1^{1/2}
	\\&
	 \le C|\bn|  
	\Big(	\sum_{i_1,i_2} \min(1,2^{-(|i_1-i_2|-|\bn|)/2}) \norm{h_1^{i_1}}_1 \Big)^{1/2}
	\Big(	\sum_{i_1,i_2} \min(1,2^{-(|i_1-i_2|-|\bn|)/2}) \norm{h_2^{i_2}}_1 \Big)^{1/2}
	\\&
	 \le C|\bn|^2
	\Big(	\sum_{i_1} \norm{h_1^{i_1}}_1 \Big)^{1/2}
	\Big(	\sum_{i_2} \norm{h_2^{i_2}}_1 \Big)^{1/2}
	\\&
	 = C|\bn|^2
	 \prod_{j=1}^2 \norm{h_j}_1^{1/2}.
\end{align*}
\end{proof}

\subsection{Contribution of $(h_1,g_2)$}

Redefine $\scripte'\subset\reals^1$ to be the union over $\beta_1$ of all
intervals $I_{1,\beta_1}^*$. This is a subset of the set $\scripte$
used above, so $|\scripte'|= O(\alpha^{-1/2})$. Note that $\norm{Q_{k+n_2}g_2}_\infty = O(\alpha^{1/2})$
uniformly in $k,n_2$. A simplified modification of the reasoning in \S\ref{subsection:h1h2},
using the simple inequality $\norm{\scriptb(f,1)}_1\le C\norm{f}_1$ 
and taking this bound $\norm{Q_{k+n_2}g_2}_\infty = O(\alpha^{1/2})$
into account, gives
\[ \int_{\reals\setminus\scripte'} |M_\bn(h_1,g_2)| \le C\alpha^{1/2} \norm{h_1}_1 \le C\alpha^{1/2}.\]
Therefore
\[ \big|\big\{x\in\reals\setminus\scripte': |M_\bn(h_1,g_2)|>\tfrac14\alpha\big\}\big|
\le C\alpha^{1/2}\alpha^{-1} = C\alpha^{-1/2}.\]
Since also $|\scripte'| = O(\alpha^{-1/2})$,
\[ \big|\big\{x\in\reals: |M_\bn(h_1,g_2)|>\tfrac14\alpha\big\}\big|
= O(\alpha^{-1/2}),\] as desired. 

The contribution of $(g_1,h_2)$ is handled in the same way, interchanging the roles of
the two indices. The contribution
of $(g_1,g_2)$ is $O((c_0\alpha^{1/2})^2)$ in $L^\infty$ norm,
hence is $<\tfrac14\alpha$ almost everywhere if $c_0$ is chosen to be sufficiently small.
That completes the proof of Lemma~\ref{lemma:Lone}.
\qed

\section{The full supremum} \label{section:full}

In this section we prove Theorem~\ref{thm:full}, concerning $M^{\text{full}}$ and $M^{\text{full}}_0$.
Despite the resemblance between $M^{\text{full}}$ and the linear spherical maximal function
\[ Mf(x) = \sup_{r\in(0,\infty)} \int_{S^{d-1}} |f(x+ry)|\,d\sigma(y),\]
where $f$ is defined on $\reals^d$ and $x\in\reals^d$,
inequalities for $M^{\text{full}}$ can be obtained 
in a straightforward way using the Hardy-Littlewood-maximal function in $\reals^1$ together with H\"older's inequality.


The restrictions on the exponents in
Theorem~\ref{thm:full} are necessary, except possibly for certain endpoint
cases which we will not discuss.
Firstly, necessity of the relation $q^{-1} = p_1^{-1}+p_2^{-1}$ is 
is an immediate consequence of a scaling symmetry. 
Secondly, it is necessary that $p_j\ge 1$ for both indices $j$.
Thirdly, for boundedness of $M_0^{\text{full}}$,
and hence for boundedness of the larger operator $M^{\text{full}}$,
it is necessary that $q(p_1,p_2)>1$. 

To see that the inequality cannot hold when $q(p_1,p_2)<1$,
let $\delta\in(0,1]$ be a small parameter, and
choose $f_j$ to be indicator functions of intervals of lengths $\delta$ centered at $0$. 
Let $q = q(p_1,p_2)$.  
Then $\sup_{1<r<2}\int_{S^1} f_1(x+r\cos(\theta))\,f_2(x+r\sin(\theta))\, d\theta$ 
is bounded below by $c \delta$ for all $x\in [\tfrac54,\tfrac74]$.  
Thus 
\[ \frac{ \|M^{\text{full}}_0(f_1,f_2)\|_q}{ \|f_1\|_{p_1}\|f_2\|_{p_2}}
\gtrsim \delta^{1-p_1^{-1}-p_2^{-1}} = \delta^{1-q^{-1}}.\] 
If $q<1$, this quantity tends to infinity as $\delta$ tends to zero.

\begin{proof}[Proof of Theorem~\ref{thm:full}]
Let $\HLM$ denote the Hardy-Littlewood maximal function for $\reals^1$.
For any $s>0$ define $\HLM_s(f)  = \HLM(|f|^s)^{1/s}$.

Let $\gamma(t) = (\gamma_1(t),\gamma_2(t)) = (\cos(t),\sin(t))$.
Assume that $k\tfrac\pi2$ does not belong to the support of the cutoff function $\eta$
fr any $k\in\{0,1,2,3,4\}$. Then
$d\gamma_j/dt$ does not vanish on the support of $\eta$, for $j=\in\{1,2\}$.
Consequently $|\scriptb(f_1,f_2)| \le CM(f_1)\norm{f_2}_{L^\infty}$,
and likewise with the roles of the indices reversed. Therefore
$M^{\text{full}}_0$ maps $L^1\times L^\infty$ and $L^\infty\times L^1$ to weak $L^1$.
Interpolating with the trivial $L^\infty\times L^\infty \to L^\infty$ bound, we 
conclude that $M^{\text{full}}_0$ maps $L^{p_1}\times L^{p_2}$ to $L^{q(p_1,p_2)}$
whenever $(p_1^{-1},p_2^{-1})\in F_0$.

To treat $M^{\text{full}}$, introduce intervals 
\[ I^k_n = \{\theta\in [0,2\pi]: |\theta-\tfrac{k\pi}{2}|\in[2^{-n},2^{-n+1}]\]
where $k\in\{1,2,3,4\}$ and $n\in\naturals$ is large.
We will analyze only $k=1$; the same reasoning will apply to $k=1,2,3,4$.
Write $I_n = I^0_n$. Define
\[\scriptm_n(f_1,f_2)(x) = \sup_{r>0} \int_{I_n} |f_1(x+r\cos(\theta)) f_2(x+r\sin(\theta)|\,d\theta.\]

\begin{lemma} \label{lemma:cauchy-schwarz}
\begin{equation} \label{lastlemma:conclusion}
\scriptm_n(f_1,f_2)(x) \le C \HLM_2f_1(x)\cdot \HLM_2f_2(x) \  \text{ uniformly in $n,x,f_1,f_2$.}
\end{equation}
For any $\tau>2$ there exists $C_\tau<\infty$ such that for all $f_1,f_2$,
\begin{equation} M^{\text{full}}(f_1,f_2)\le C_\tau \HLM_\tau(f_1)\cdot \HLM_\tau(f_2).  \end{equation}
\end{lemma}

Boundedness of $M^{\text{full}}$ from $L^{p_1}\times L^{p_2}$
to $L^{q(p_1,p_2)}$ for all $p_1,p_2>2$ is an immediate consequence of the second conclusion.
\end{proof}

\begin{proof}
By the Cauchy-Schwarz inequality,
\begin{equation*} 
\scriptm_n(f_1,f_2)(x) \le \sup_r\,
(2^n \int_{I_n} |f_1(x+r\cos(\theta))|^{2}\,d\theta)^{1/2} \cdot
(2^{-n} \int_{I_n} |f_2(x+r\sin(\theta))|^{2}\,d\theta)^{1/2}. \end{equation*}
Since $d\cos\theta/d\theta$ is bounded away from zero on $\bigcup_n I_n$, 
the first factor on the right-hand side satisfies
\[ \sup_r\, (2^n \int_{I_n} |f_1(x+r\cos(\theta))|^{2}\,d\theta)^{1/2} \le C \HLM_{2}f_1(x),\] 
uniformly in $r\in(0,\infty)$.
In the second factor, 
substitute $\sin(\theta) = 1-s^2$ with $|s|\asymp 2^{-n}$. $2^{-n}\,d\theta$ is comparable to $ds$
so the second factor is majorized by
\begin{align*}
 C\sup_r\, \left(\int_{[2^{-2n},2^{-2n+2}]}|f_2(x+r+rs)|^2\,ds\right)^{1/2}
 \le C\sup_r\, \left(\int_{|s|\leq 2}|f_2(x+rs)|^2\,ds\right)^{1/2},
\end{align*}
uniformly in $r$.

The second conclusion follows in the same way by summation over $n$, since 
\[\int_{I_n}|f_2(x+r\sin(\theta))|^2\,d\theta 
\lesssim 2^{-n(1-(2/\tau))} \big( \int |f_2(x+r\sin(\theta))|^\tau\,d\theta \big)^{2/\tau}\]
by H\"older's inequality.
\end{proof}

\section{Proof of the trilinear smoothing inequality}

\subsection{Convergence of the integral} \label{subsect:convergence}

Define the degenerate locus $\Sigma\subset U$
to be the union, over all distinct pairs of indices $i\ne j \in \{1,2,3\}$,
of $\{x: \det(\nabla\varphi_i(x),\,\nabla\varphi_j(x))=0\}$.
This set $\Sigma$ is a real analytic variety of positive codimension in $U$.
By \L{}ojasiewicz's theorem there exist $\kappa_1,\kappa_2,c\in(0,\infty)$ such that
\begin{equation}
|\det(\nabla\varphi_i(x),\,\nabla\varphi_j(x))|
\ge c\dist(x,\Sigma)^{\kappa_1}\ \forall\, i\ne j\in\{1,2,3\},\ \forall\, x\in U
\end{equation}
and
\begin{equation} \label{smallnbdsmall}
\big|\big\{x\in U: \dist(x,\Sigma)<\delta\big\}\big| = O(\delta^{\kappa_2})\ \ \forall\,\delta>0.
\end{equation}

Let $\delta>0$ be small. Cover $U_\delta = \{x\in U: \dist(x,\Sigma)>\delta\}$
by cubes $Q_m$ of sidelength $\delta^{2\kappa_1}$ that are pairwise disjoint,
except for possible intersections along their boundaries,
and that intersect $U_\delta$.
Then by \eqref{smallnbdsmall},
\[ \int_{\reals^2\setminus \bigcup_m Q_m} \prod_j |f_j\circ\varphi_j|\,|\eta|\,dx
= O(\delta^{\kappa_2} \prod_j \norm{f_j}_{L^\infty}).\]
On the other hand, a simple change of variables shows that
for any  permutation $(i,j,k)$ of $(1,2,3)$,
\[ \int_{Q_m} \prod_l |f_l\circ\varphi_l|\cdot|\eta|\,dx
= O(\delta^{-\kappa_1}\norm{f_i}_{L^1}\norm{f_j}_{L^1} \norm{f_k}_\infty).\]
There are at most $O(\delta^{-4\kappa_1})$ cubes $Q_m$.
Therefore 
for all $\delta>0$,
\[ |\scriptt(\bff)| \le C\delta^{\kappa_2} \prod_{l=1}^3 \norm{f_j}_{L^\infty}
+ C\delta^{-C\kappa_1} \norm{f_i}_{L^1}\norm{f_j}_{L^1} \norm{f_k}_\infty.\]
It follows from real interpolation that there exists $q<\infty$ such that
\[ \int \prod_j (f_j\circ\varphi_j)\cdot |\eta|\,dx \le C\prod_{j=1}^3 \norm{f_j}_{L^q}\]
for all nonnegative functions $f_j\in L^\infty$.
Therefore the integral defining $\scriptt(\bff)$ converges absolutely whenever each $f_j\in L^q$.


 \subsection{Two lemmas}

Our qualitative main hypothesis implies a formally stronger quantitative reformulation.
For each index $j$, there exists a $C^\omega$ vector field
$V_j$ in $U$ that satisfies $V_j(\varphi_j)\equiv 0$,
but does not vanish identically. Fix one such vector field for each index. 
Define $h_j = V_3\varphi_j\in C^\omega$, and define 
\begin{equation*} 
\scriptl(g_1,g_2) = \sum_{j=1}^2 h_j\cdot (g_j\circ\varphi_j). 
\end{equation*}
For any $C^1$ functions $f_j$,
\begin{equation*} V_3\sum_{j=1}^3 (f_j\circ\varphi_j) = \scriptl(f'_1,f'_2)
\end{equation*}
where $f'$ denotes the derivative of $f$.
The transversality hypothesis
that $\det(\nabla\varphi_j,\nabla\varphi_3)$ does not vanish identically in $U$, 
is equivalent to the condition that neither coefficient function $h_j$ vanishes identically.

\begin{lemma}  \label{lemma:jets}
There exist $M_0\in\naturals$ 
and a real analytic variety $\Sigma'\subset U$ of positive codimension
with the following property.
For any $M\ge M_0$ 
there exist $C<\infty$ and $\kappa<\infty$,
such that for any $z\in U$, any neighborhood $U'$ of $z$, 
and any functions $F_1,F_2$ that belong to $C^M(U')$,
\begin{equation} \label{jetted}
\dist(z,\Sigma')^\kappa \sum_{j=1}^2 \sum_{1\le k\le M} 
|F_j^{(k)}(\varphi_j(z))| 
\\ \le  C \sum_{1\le |\alpha|\le M} 
\Big| \frac{\partial^\alpha}{\partial x^\alpha} \scriptl(F_1,F_2)(z) \Big|.
\end{equation}
\end{lemma}

Here $F^{(k)}$ denotes the $k$-th derivative of $F$.

\begin{proof}
For each $z\in U$ there is a natural linear mapping $L_{M,z}$ from pairs of $M$-jets
of $(F_1,F_2)$ at $(\varphi_1(z),\varphi_2(z))$ satisfying $F_j(\varphi_j(z))=0$,
to $M$-jets of functions vanishing at $z$, defined by interpreting the mapping
$(F_1,F_2)\mapsto \sum_{j=1}^2 h_j(z)\,(F_j\circ\varphi_j)(z)$ as a mapping of $M$-jets. 
It is shown in \cite{foursublevel}
that if every $C^\omega$ solution $(F_1,F_2)$ of $\sum_j h_j\cdot (F_j\circ\varphi_j)=0$
in any nonempty open set vanishes identically,
such that for every $z$ satisfying $h_1(z)\ne 0$ and $h_2(z)\ne 0$,
there exists $\overline{M}$ such that $L_{M,z}$ is injective for every $M\ge \overline{M}$.

This hypothesis on solutions of $\sum_j h_j\cdot (F_j\circ\varphi_j)=0$ is satisfied. 
Indeed, given a solution $(F_1,F_2)$ in an open neighborhood of $z$, 
construct $G_j$ to satisfy $G'_j=F_j$ in a neighborhood of $\varphi_j(z)$.
Then $V_3\sum_{j=1}^2 (G_j\circ\varphi_j)\equiv 0$.
Therefore in any sufficiently small subball in which $V_3$ does not vanish, 
there exists a $C^\omega$ solution $G_3$ of $\sum_{j=1}^3 (G_j\circ\varphi_j)\equiv 0$.
Therefore $G_1,G_2,G_3$ must be locally constant by our main hypothesis, 
whence $F_1,F_2$ vanish identically. 

Consider any $z_0$ such that $h_j(z_0)$ is nonzero for both indices $j=1,2$. 
Consider any $M\in\naturals$ for which $L_{M,z_0}$ is injective.
Representing $L_{M,z}$ by a matrix $\mbbm(z)$ with real analytic coefficients,
the determinant $\det(\mbbm(z)^*\circ\mbbm(z))$ is nonzero at $z_0$.
This determinant is a real analytic function of $z$.
Therefore the set on which it vanishes is a real analytic subvariety
$\Sigma'\subset U$ of positive codimension.
By \L{}ojasiewicz's theorem,
this determinant is bounded below by $c\dist(z,\Sigma')^\kappa$
for some $\kappa,c>0$.
Therefore \eqref{jetted} holds.
\end{proof}

\begin{lemma} \label{lemma:stronghypothesis}
Let $N,d\ge 1$.
There exist $\tau>0$ and $A,\sigma<\infty$ with the following property.
Let $Q\subset\reals^d$ be the closed unit cube. Let $f:Q\to\complex$ be a $C^{N+1}$ function,
Define
\begin{equation} 
\delta = \min_{x\in Q} \sum_{0\le |\alpha|\le N} \left|\frac{\partial^\alpha f}{\partial x^\alpha}(x)\right|
\end{equation}
and $B = 1 + \sum_{|\alpha| = N+1} \norm{\partial^\alpha f}_{C^0(Q)}$.
For each $\eps>0$ let $S(f,\eps) = \{x\in Q: |f(x)| \le \eps\}$.  Then
\begin{equation} |S(f,\eps)| \le C \eps^\tau \delta^{-\sigma} B^A.  \end{equation}
\end{lemma}

\begin{proof}
We will use this fact:
For any $d,N$ there exist positive constants $\sigma,C$
such that for any polyomial $R$ of degree $N$
satisfying $\sum_{0\le|\beta|\le N} |\partial^\beta R(0)|\ge 1$,
\begin{equation} \label{Rsublevel} |\{x\in Q: |R(x)|\le\eps\}| \le C\eps^{\sigma}.\end{equation}
The case $d=1$ is well known, and the higher-dimensional case is a simple consequence.

Let $B = 1 + \max_{|\alpha| = N+1} \norm{\partial^\alpha f}_{C^0}$.
Let $c_0>0$ be a small constant. Define $r\in(0,1]$ to satisfy $Br^{N+1} = c_0\eps$.
Partition $Q$ into $O(r^{-d})$ subcubes $Q_i$, each of sidelength $r$.

On each cube $Q_i$, express \[f = P_i + O(B r^{N+1}) = P_i + O(c_0\eps)\]
where $P_i$ is the Taylor  polynomial of degree $N$ for $f$ at the center $x_i$ of $Q_i$.
If $c_0$ is chosen sufficiently small then $|f(x)-P_i(x)|\le\eps/2$ for every $x\in Q_i$.
Thus it suffices to bound
$\sum_i |\{x\in Q_i: |P_i(x)|\le \tfrac12\eps\}|$.

Now $\sum_{|\beta|\le N} |\partial^\beta P(x_i)|\ge \delta$
since this inequality holds for $f$ and since $P$ is the Taylor polynomial
of degree $N$ for $f$ at $x_i$.
Applying the inequality \eqref{Rsublevel} to
\[ R(x) = \delta^{-1}r^{-N} P_i(r\cdot (x+x_i))\] and incorporating a Jacobian factor
$r^d$ resulting from this change of variables yields
\[ |\{x\in Q_i: |P_i(x)\le\eps/2  \}| = O(r^d\cdot (r^{-N} \delta^{-1} \eps)^{\sigma}). \]
Therefore by summing over $O(r^{-d})$ cubes $Q_i$ we obtain the bound
\[ |\{x\in Q: |f(x)|\le\eps\}| 
= O((r^{-N} \delta^{-1}\eps)^\sigma)  
= O\big(\big[(B^{-1}\eps)^{-N/(N+1)} \cdot \delta^{-1}\eps\big]^\sigma\big)  
= O(B^{A} \delta^{-\sigma} \eps^\tau) \]
with $A = N\sigma/(N+1)$ and $\tau = \sigma/(N+1)>0$.
\end{proof}

\subsection{Reduction to a sublevel set inequality in the nondegenerate case }


We say that $\Phi = (\varphi_1,\varphi_2,\varphi_3)$ is nondegenerate if
for every pair of distinct indices $i,j$, 
$\det(\nabla\varphi_i(x),\nabla\varphi_j)$
vanishes nowhere in $U$, 
and moreover, the mapping $x\mapsto (\varphi_i(x),\varphi_j(x))$
is a diffeomorphism of $U$ onto an open subset of $\reals^2$.
For the present, we assume $\Phi$ to be nondegenerate.
The general case will be treated in \S\ref{subsect:degen}.

Let $\lambda\ge 1$.
It suffices to show that there exist $\tau>0$ and $C<\infty$ such that
if each $\widehat{f_j}(\xi)$ is supported where $|\xi|= O(\lambda)$, 
and where at least one Fourier transform 
$\widehat{f_k}(\xi)$ is supported where $|\xi|\gtrsim\lambda$, then
\begin{equation}
|\scriptt(\bff)| \le C\lambda^{-\tau} 
	\ \text{ if $\norm{f_j}_{L^\infty}\le 1$ for every index $j$}.
\end{equation}

Fix an exponent $\gamma\in(\tfrac12,1) $. Partition $\reals^1$
into intervals $I_m$ of sidelengths $\lambda^{-\gamma}$. 
Denote by $I_m^*$ the interval of length $3\lambda^{-\gamma}$ with the same center as $I_m$.
Decompose each $f_j$ as
\begin{equation} \label{localFourier}
f_j(y) = \sum_{m\in\integers} \eta_m(y) \sum_{k\in\integers} a_{j,m,k} e^{i\pi \lambda^\gamma ky}
\end{equation}
with each $\eta_m\in C^\infty_0(\reals)$ supported on $I_m^*$ 
satisfying $\norm{\eta^{(n)}_m}_{C^0} \le C_n \lambda^{n\gamma}$ for every $n\ge 0$, where 
$\eta^{(n)}$ denotes the $n$-th derivative of $\eta$.
By Bessel's inequality and the bound $\norm{f_j}_{L^\infty} = O(1)$,
the coefficients satisfy
\begin{equation} \sum_{k\in\integers} |a_{m,k}|^2 = O(1) \end{equation}
uniformly in all parameters.

Express
\begin{equation}
	\scriptt(\bff)
	= \sum_{\bm} \sum_{\bk} \prod_{j=1}^3 a_{m_j,k_j} I(\bm,\bk)
\end{equation}
where $\bm = (m_1,m_2,m_3)$, $\bk = (k_1,k_2,k_3)$, and
\begin{equation}
I(\bm,\bk) = \int_{\reals^2} e^{i\pi\lambda^\gamma \sum_{j=1}^3 k_j\varphi_j(x)}
\,\prod_{j=1}^3 \eta_{m_j}(\varphi_j(x))\,\eta(x)\,dx.
\end{equation}
We say that $\bm = (m_1,m_2,m_3)$ is interacting if there exists $x\in U$ such that $\varphi_j(x)\in I_{m_j}$
for each $j\in\{1,2,3\}$.
For each interacting tuple $\bm$, choose a point $x_\bm$ satisfying these inclusions.
Let $\rho>0$ be a small exponent.
$I(\bm,\bk)=0$ for all $\bk$ if $\bm$ is not interacting, so
it suffices to sum only over interacting $\bm$. 
There are $O(\lambda^{2\gamma})$ such ordered triples $\bm$.

It follows from integration by parts that there exists $c>0$ such that for every interacting $\bm$,
\begin{equation} \label{nonstationarydecay} 
|I(\bm,\bk)| \le C_N \lambda^{-N}\ \forall N<\infty \end{equation}
if
\begin{equation} \label{stationary}
\big|\nabla(\sum_{j=1}^3 k_j\varphi_j(x_\bm))\big| \ge\lambda^\rho.  \end{equation}
The constants $C_N$ depend also on $\rho$.
The proof of \eqref{nonstationarydecay} exploits
the quantitative hypothesis on the Fourier support of each factor $f_j$,
second order Taylor expansion of $\varphi_j$ about $x_\bm$,
and the restriction $\gamma>\tfrac12$.

For any interacting $\bm$ and any $k_1$,
there are $O(\lambda^{2\rho})$ ordered pairs $(k_2,k_3)$
such that $\bk$ satisfies the stationarity condition \eqref{stationary}.
This holds for any permutation of the indices $1,2,3$.

Consider an arbitrary interacting $\bm$, and denote by $\sum^*_\bk A_{\bk}$
the sum of some quantities $A_{\bk}$
over all $\bk = (k_1,k_2,k_3)$ that satisfy the stationarity condition \eqref{stationary}. 
By taking the supremum over one of the three indices $k_1,k_2,k_3$
and applying the Cauchy-Schwarz inequality with respect to the other two indices,
doing this for all three possible choices over the first index chosen,
and taking the geometric mean of the resulting three inequalities, we get
\begin{equation} \label{extrafactor}
\begin{aligned}
	\sum^*_{\bk} \prod_{j=1}^3 |a_{m_j,k_j}|
	&\le C\prod_{j=1}^3 ((\sum_k |a_{m_j,k}|^2)^{1/2})^{2/3} (\sup_k |a_{m_j,k}|)^{1/3}
	\\&
	\le C \prod_{i=1}^3 \norm{f_i}_{L^\infty}^{2/3}\cdot\prod_{j=1}^3 \sup_k |a_{m_j,k}|^{1/3}.
\end{aligned} \end{equation}

Let $\sigma>0$ be a small exponent.
Decompose each $f_j$
as $f_j = g_j+h_j$
where each $h_j$ has an expansion of the form \eqref{localFourier}
with each coefficient satisfying $|a_{m_j,k}|\le \lambda^{-\sigma}$,
while each $g_j$ has an expansion of the form \eqref{localFourier}
with at most $O(\lambda^{2\sigma})$ nonzero coefficients $a_{m_j,k}$ for each $m_j$.

Expand $\scriptt(\bff)$ as a sum of eight terms by expressing 
$f_j = g_j+h_j$ for each index $j$ in the integral defining $\scriptt(\bff)$.
Seven of those terms involve at least one $h_j$. For each of those seven terms
we obtain an upper bound $O(\lambda^{-\sigma/3})$ from 
\eqref{extrafactor} and the assumption $|a_{m_j,k}| \le \lambda^{-\sigma}$.

There remains $\scriptt(\bg) = \scriptt(g_1,g_2,g_3)$. 
Each $g_j$ can be expressed as a sum of $O(\lambda^{2\sigma})$ terms,
each of which takes the local exponential monomial form
\begin{equation} \label{monomialform}
	g_j(y) = \sum_m a_{j,m} e^{i\pi\lambda^\gamma k_{j,m}y}\eta_m(y)
\end{equation}
with $a_m = O(1)$, for some $k_{j,m}\in\integers$.
Thus we have shown that
if each $f_j$ has the indicated Fourier support and satisfies $\norm{f_j}_{L^\infty} = O(1)$ then
\begin{equation} \label{3termbound}
|\scriptt(\bff)| \le C_{N,\rho} \lambda^{-N} + C\lambda^{C\rho} \lambda^{-\sigma/3}
+ C\lambda^{6\sigma} \sup_{\bg}|\scriptt(\bg)|
\end{equation}
for any $N<\infty$ and any $\rho,\sigma>0$,
where the supremum is taken over all $\bg$ of local exponential monomial form \eqref{monomialform}.
It suffices to prove that there exists $\tau>0$ such that
$|\scriptt(\bg)| = O(\lambda^{-\tau})$ uniformly for all such $\bg,\lambda$,
for then upon choosing $\sigma>0$ to be sufficiently small as a function of $\tau$,
and then $\rho>0$ sufficiently small as a function of $\sigma$, 
\eqref{3termbound} becomes the desired upper bound for $\scriptt$.

Now assume that $\widehat{f_1}(\xi)$ is supported where $|\xi|\gtrsim\lambda$.
It is elementary to verify, using integration by parts, that the $\ell^2$ norm, with respect to $k$, 
of those coefficients $a_{1,m,k}$ satisfying $|k|\le \lambda^{1-\gamma-\rho}$
is $O(\lambda^{-c\rho})$ for some $c>0$.
Consequently we may assume that $g_1$ satisfies 
\begin{equation} |k_{1,m}|\ge \lambda^{1-\gamma-\rho} \ \forall\, m.  \end{equation}

Classify interacting tuples $\bm$ into two types. Say that $\bm$ is stationary if
\begin{equation} 
\big|\nabla(\sum_{j=1}^3 k_{j,m_j} \varphi_j(x_\bm))\big| <\lambda^\rho,  \end{equation}
and nonstationary otherwise. 
Expand $\scriptt(\bg)$ as a sum of contributions of interacting ordered triples $\bm$, as above.
By \eqref{nonstationarydecay}, the sum of the contributions of all nonstationary 
indices $\bm$ to $\scriptt(\bg)$ is $O_N(\lambda^{-N})$ for every $N<\infty$. 

To complete the proof,
it suffices to show that the sum of the Lebesgue measures of the supports of $\eta_\bm$,
summed over all stationary interacting triples $\bm$,
is $O(\lambda^{-\tau})$ for some $\tau>0$.
In doing this, it suffices to treat those $\bm\in\integers^3$ that lie in an arbitrary congruence class
modulo $3\integers^3$, since there are finitely many such congruence classes. Fix any such class henceforth,
and restrict attention to $\bm$ lying in that class.

Define 
\begin{equation}
F_j(y) = \sum_{m_j}\lambda^{-(1-\gamma-\rho)} k_{j,m_j}\,\one_{I_{m_j}^*}(y),
\end{equation}
noting that the restriction of $m_j$ to a congruence class modulo $3$
ensures that the intervals $I_{m_j}^*$ are pairwise disjoint.
These functions satisfy 
$|F_1(y)|\gtrsim 1$ for all $y$.
The union over all interacting stationary $\bm$ of the set of all $x\in\reals^2$
satisfying $\varphi_j(x)\in I_{m_j}^*$ for all three indices $j$
is contained in
\begin{equation} \label{sublevel1}
\sS = \big\{x\in U: \big|\sum_{j=1}^3 (F_j\circ\varphi_j)(x)\nabla\varphi_j(x)\big| 
= O(\lambda^{-(1-\gamma)+C\rho}) \big\}.
\end{equation}

By setting $\eps = \lambda^{-1+\gamma+C\rho}$
and taking the inner product of 
$\sum_{j=1}^3 (F_j\circ\varphi_j)(x)\nabla\varphi_j(x) $ with the vector $V_3(x)$
we find that
\begin{equation}
\big| h_1(x)(F_1\circ\varphi_1)(x) + h_2(x)(F_2\circ\varphi_2)(x)\big| = 
O(\eps) \ \ \forall\,x\in\sS.
\end{equation}
Therefore to conclude the analysis of the nondegenerate case, it suffices 
the next lemma.

\begin{lemma} \label{lemma:sublevel}
There exist $\tau>0$ and $C<\infty$ such that for every $\eps>0$
and every ordered pair $(F_1,F_2)$ of Lebesgue measurable functions,
\begin{equation} \label{sublevel2}
\big|	\{x\in	S(F_1,F_2,\eps): 
|F_1\circ\varphi_1(x)|\ge 1 \}\big|  \le C \eps^\tau.
\end{equation}
\end{lemma}

Redefine $F_1(y)$ whenever $|F_1(y)|<1$, so that $|F_1(y)|\ge 1$ for every $y$.
We aim to show that 
\begin{equation} \label{weaim} |S(F_1,F_2,\eps)| = O(\eps^\tau) \end{equation} 
for some constant $\tau>0$.  Set $\sS =  S(F_1,F_2,\eps)$.
Let $\tau_0$ be a parameter to be chosen below.
If $|\sS|\le \eps^{\tau_0}$ then the proof is complete.
We assume henceforth that $|\sS| > \eps^{\tau_0}$.

\subsection{Proof of Lemma~\ref{lemma:sublevel}}

After introducing a smooth finite partition of unity in the original integral \eqref{trilinearform}
and restricting attention to each summand individually, we may
change variables so that $x = (x_1,x_2)$ with $\varphi_j(x) = x_j$ for $j=1,2$.
Consider \[H(x) = H(x_1,x_2) = h_2(x)/h_1(x),\] 
which is a nonvanishing real analytic function in the nondegenerate case.
Write 
\begin{align*}
	H_{x_2}(x_1) &= H(x_1,x_2) 
	\\	H'_{x_2}(x_1) &= \frac{\partial H(x_1,x_2)}{\partial x_1}.
\end{align*}
Since any  $C^\omega$ solution $(F_1,F_2)$ of $h_1(x) F_1(x_1)
+h_2(x) F_2(x_2)=0$ in any nonempty connected open set must vanish identically,
$H$ cannot be a function of either $x_1$ alone, or of $x_2$ alone.
Thus $\partial H/\partial x_1$ does not vanish identically.
Therefore there exist $C,c\in\reals^+$
such that for any $\delta>0$,
$|\{x\in U: |\partial H(x)/\partial x_1|<\delta\}| \le C\delta^c$.
Consequently there exist $C, c$ and $\sS_0\subset\sS$
such that
\begin{equation} \label{H'bound} 
|\frac{\partial H}{\partial x_1} (x)| \ge c|\sS|^C \ge c\eps^{C\tau_0}
\ \text{ for every $x\in \sS_0$.}
\end{equation}

We claim that there exist $\tilde\sS\subset\sS_0$ satisfying 
$|\tilde\sS|\gtrsim |\sS_0|^C \gtrsim |\sS|^{C'}$
and $C^\infty$ functions $G_j$ with $\norm{G_j}_{C^N} = O_N(\lambda^{C\rho})$
for every $N\in\naturals$, such that
\[ (F_j-G_j)(y) = O(\eps) 
\ \text{ for every $y\in\varphi_j(\tilde\sS)$.} \]
To prove this claim, we say that $x\in\sS_0$ is rich if the set of all $x'\in \sS_0$ satisfying
$\varphi_1(x')=\varphi_1(x)$ has one-dimensional measure $\gtrsim |\sS_0|$. 
The set $\sS_1$ of all rich points of $\sS_0$ has measure $\gtrsim |\sS_0|$.
There exists $\barx_2$ such that the set $S$ of all $x_1$ satisfying $(x_1,\barx_2)\in \sS_1$
has one-dimensional Lebesgue measure $\gtrsim |\sS_1|$.
For any $y\in S$, 
\begin{equation} \label{F1structure}
F_1(y) = G_1(y) + O(\eps)
\end{equation}
with
\begin{equation} G_1(y) = - c_1H(y,\barx_2) \end{equation}
where $c_1 = F_2(\barx_2)$.  Thus $c_1 = O(\eps^{-C\rho})$.
Since points in $\sS_1$ are all rich, $\sS_2 = \{x\in \sS_1:
\varphi_1(x)\in S\}$ satisfies
$|\sS_2| \gtrsim |\sS_1|^2\gtrsim |\sS|^2$, and \eqref{F1structure}
holds for every $y\in \varphi_1(\sS_2)$.
Thus
\begin{equation} \label{F1structure2}
F_1\circ\varphi_1(x) = G_1\circ \varphi_1(x) + O(\eps)
\qquad\forall\,x\in \sS_2.
\end{equation}

Repeating this reasoning for $F_2$, with $\sS$ replaced by the subset $\sS_2$, gives
$F_2\circ\varphi_2(x) = G_2\circ\varphi_2(x) + O(\eps)$
for all $x\in \sS_3$,
with $\sS_3\subset\sS_2\subset \sS$ satisfying $|\sS_3|\gtrsim |\sS|^4$,
and with $G_2$ defined in the same way that $G_1$ was defined, with the roles
of the indices $j=1,2$ reversed.
Thus $\tilde\sS = \sS_3$ and these functions $G_1,G_2$ have the properties claimed.
Moreover, by \eqref{H'bound}, the derivative of $G_1$ satisfies 
\begin{equation} \label{G'lowerbound} 
|G_1'(y)| \gtrsim \eps^{C\tau_0} \ \ \forall\,y\in \varphi_1(\tilde\sS).
\end{equation}

It suffices to show that
\[|\{x\in B: \scriptl(G_1,G_2)(x)= O(\eps)\}| = O(\eps^c) \ \text{ for some $c>0$.} \]
Define \[\scriptg(x) = \sum_{j=1}^2 h_j(x)\,(G_j\circ\varphi_j)(x) = \scriptl(G_1,G_2)(x).\]
We have $\norm{\scriptg}_{C^N} = O_N(\lambda^{C\rho})$ for every $N\ge 0$.
By its construction,
$c_2=F_1(y)$ for some $y$ in the domain of $F_1$, so $|c_2|\gtrsim 1$ by the
hypothesized lower bound for $|F_1|$.

Therefore according to Lemma~\ref{lemma:jets},
\begin{equation} 
	\eps^{C\tau_0} \lesssim
	|G'_1\circ\varphi_1(x)|
	\le \sum_{j=1}^2 \sum_{k=1}^M |G_j^{(k)}\circ\varphi_j(x)|
\le  C' \sum_{1\le |\alpha|\le M} 
\big| \frac{\partial^\alpha}{\partial x^\alpha} \scriptg(x)\big|\ \ \forall\,x\in \tilde\sS.
\end{equation}
This lower bound for $\scriptg$ allows us to apply
Lemma~\ref{lemma:stronghypothesis} to it, yielding 
\[ |\tilde\sS| \le C \eps^\tau (\eps^{C\tau_0})^{-\sigma} \eps^{C\rho B}\]
for certain exponents $\tau,\sigma,A\in\reals^+$.
If $\tau_0,\rho$ are chosen to be sufficiently small then this bound takes the form
$|\tilde\sS|= O(\eps^{\tau/2})$. 
Since $|\tilde\sS|\gtrsim |\sS|^C$,
$|\sS|= O(\eps^{\tau/2C})$, completing the proof of \eqref{weaim}.



\subsection{The degenerate case} \label{subsect:degen}

Let $\lambda\in\reals^+$ be large.
Let $\delta>0$ be a small parameter that depends on $\lambda$, and is to be chosen below.
Let $\Sigma$, $c$, $\kappa_1$, $\kappa_2$, $Q_m$ be as in \S\ref{subsect:convergence}.
$Q_m$ are cubes of sidelengths $\delta^{2\kappa_1}$.

The restriction of $\Phi$ to each cube $Q_m$ defines a nondegenerate datum,
though this nondegeneracy is not uniform in $\lambda$.
The above analysis demonstrates that there exist an exponent $\tau\in(0,\infty)$
and a constant $C\in(0,\infty)$ that
depend on $\Phi$, but not on $\lambda,\delta$, such that
\begin{equation} \label{quantitativedegeneration}
\big| \int_{Q_m} \prod_j (f_j\circ\varphi_j)\,\eta\,dx\big|
\le C\delta^{-C\kappa_1} \lambda^{-\tau} \prod_j \norm{f_j}_{L^\infty}
\end{equation}
uniformly for all functions $f_j\in L^\infty$,
provided that $\lambda^{-\gamma}\lesssim \delta^{2\kappa_1}$,
that is, provided that $\delta\gtrsim \lambda^{-\gamma/2\kappa_1}$.

The proof of \eqref{quantitativedegeneration}
is simply a matter of tracing the dependence on $\delta$ of constants in the above analysis.
For instance, for any $i\ne j$, the mapping $x\mapsto (\varphi_i(x),\varphi_j(x))$
with domain $Q_m$ is uniformly smooth and invertible, with inverse that is $O(\delta^{-CN\kappa_1})$
in $C^N$ norm for any $N$.

Assuming that $\norm{f_j}_{L^\infty}\le 1$  and 
that $\delta\gtrsim \lambda^{-\gamma/2\kappa_1}$, we obtain 
\begin{equation}
|\scriptt(\bff)| \lesssim \delta^{\kappa_2} 
+ \delta^{-4\kappa_1}\delta^{-C\kappa_1}\lambda^{-\tau}
\end{equation}
by majorizing the contribution of each of 
$O(\lambda^{-4\kappa_1})$ cubes $Q_m$ using \eqref{quantitativedegeneration}
and using the bound \eqref{smallnbdsmall} to control the contribution of 
the complement of their union.
The exponents $\tau,\kappa_1,\kappa_2$ and constant $C$ are independent of $\lambda,\delta$.
Choosing $\delta = \lambda^{-r}$ for a sufficiently small exponent $r$ produces the desired upper bound.
\qed


\begin{thebibliography}{20}

\bibitem {barrionuevo+etal} 
J.~A.~Barrionuevo, L.~Grafakos, D.~He, P.~Honz\'{\i}k, and L.~Oliveira,
{\em Bilinear spherical maximal function}, 
Math. Res. Lett. 25 (2018), no. 5, 1369--1388

\bibitem{christHardy}
M.~Christ,
{\em Weak type $(1,1)$ bounds for rough operators}, Ann. of Math. (2) 128 (1988), no. 1, 19--42

\bibitem{elementary!}
\bysame,
{\em On certain elementary trilinear operators}, Math. Res. Lett. 8 (2001), no. 1-2, 43--56

\bibitem{trilinear}
\bysame,
{\em On trilinear oscillatory integral inequalities and related topics},
preprint

\bibitem{quadrilinear}
\bysame,
{\em On implicitly oscillatory quadrilinear integrals}, in preparation

\bibitem{foursublevel}
\bysame,
{\em A three term sublevel set inequality}, in preparation

\bibitem{CDR} 
M.~Christ, P.~Durcik, and J.~Roos, 
{\em Trilinear smoothing inequalities and a variant of the triangular Hilbert transform},
Adv. Math. 390 (2021), Paper No.~107863


\bibitem{dosidis_ramos}
G.~Dosidis and P.~G.~Ramos,
{\em The multilinear spherical maximal function in one dimension},
preprint,
arXiv:2204.00058v1


\bibitem {heo+hong+yang}
Y.~Heo, S.~Hong, and C.~W.~Yang,
{\em Improved bounds for the bilinear spherical maximal operators},
Math. Res. Lett. 27 (2020), no. 2, 397--434

\bibitem {jeong+lee} 
E.~Jeong and S.~Lee,
{\em Maximal estimates for the bilinear spherical averages and the bilinear Bochner-Riesz operators},
J. Funct. Anal. 279 (2020), no. 7, 108629

\bibitem{lacey}
M.~Lacey,
{\em The bilinear maximal functions map into $L^p$ for $2/3<p \le 1$},
Ann. of Math. (2) 151 (2000), no. 1, 35--57

\bibitem{kenig+stein}
C.~E.~Kenig and E.~M.~Stein,
{\em Multilinear estimates and fractional integration},
Math. Res. Lett. 6 (1999), no. 1, 1--15


\bibitem{STW}
A.~Seeger, T.~Tao, and J.~Wright,
{\em Pointwise convergence of lacunary spherical means}, 
Harmonic analysis at Mount Holyoke (South Hadley, MA, 2001), 341--351,
Contemp. Math., 320, Amer. Math. Soc., Providence, RI, 2003

\bibitem{stein+wainger_BAMS}
E.~M.~Stein and S.~Wainger,
{\em Problems in harmonic analysis related to curvature},
Bull. Amer. Math. Soc. 84 (1978), no. 6, 1239--1295


\end{thebibliography}
 \end{document}